\documentclass[a4paper,11pt]{article}
\usepackage{amssymb,latexsym}
\usepackage[english]{babel}
\usepackage{graphicx}
\usepackage{amsthm}
\usepackage{amsmath}
\usepackage{slashbox}
\usepackage{scrextend}
\usepackage{epstopdf}
\usepackage{epsfig}

\newtheorem{theorem}{Theorem}[section]

\newtheorem{corollary}[theorem]{Corollary}
\newtheorem{proposition}[theorem]{Proposition}
\newtheorem{definition}{Definition}[section]
\newtheorem{observation}[theorem]{Observation}

\theoremstyle{definition}

\begin{document}

\title{{\bf Planar 3-dimensional assignment problems with Monge-like cost 
arrays}}
\author{
\sc Ante \'{C}usti\'{c}
\thanks{{\tt custic@opt.math.tugraz.at}.
Institut f\"ur Optimierung und Diskrete Mathematik, TU Graz, 
Steyrergasse 30, A-8010 Graz, Austria}
\and
\sc Bettina Klinz\thanks{{\tt klinz@opt.math.tugraz.at}.
Institut f\"ur Optimierung und Diskrete Mathematik, TU Graz, 
Steyrergasse 30, A-8010 Graz, Austria}
\and
\sc Gerhard J. Woeginger\thanks{{\tt gwoegi@win.tue.nl}.
Department of Mathematics and Computer Science, 
TU Eindhoven, P.O. Box 513, 5600 MB Eindhoven, Netherlands}
}
\date{Preliminary version, May 2014}

\maketitle


\begin{abstract}
Given an $n\times n\times p$ cost array $C$ 
we consider the problem $p$-P3AP which consists
in finding $p$ pairwise disjoint permutations 
$\varphi_1,\varphi_2,\ldots,\varphi_p$  of $\{1,\ldots,n\}$
such that $\sum_{k=1}^{p}\sum_{i=1}^nc_{i\varphi_k(i)k}$
is minimized. For the case $p=n$ the 
planar 3-dimensional assignment problem P3AP results.

Our main result concerns the $p$-P3AP on cost arrays
$C$ that are layered Monge arrays. In a layered Monge array all
$n\times n$ matrices that result from fixing 
the third index $k$ are Monge 
matrices. We prove that the $p$-P3AP and the P3AP remain NP-hard 
for layered Monge arrays. Furthermore, we show that in the
layered Monge case there always exists an optimal solution of the
$p$-3PAP which can be represented as matrix with bandwidth
$\le 4p-3$. This structural result allows us to 
provide a dynamic programming algorithm that solves 
the $p$-P3AP in polynomial time on layered Monge arrays 
when $p$ is fixed.


\end{abstract}

\section{Introduction}\label{sec:intro}

Let $S_n$ denote the set of permutations on $\{1,\ldots,n\}$. 
Given an $n\times n$ cost matrix $M$, the classical linear assignment problem
LAP can be stated as follows
$$
\min_{\varphi\in S_n} \sum_{i=1}^n m_{i\varphi(i)}.
$$
The LAP can be generalized to three dimensions in two ways. 
The resulting two problems are the axial 3-dimensional
assignment problem and the planar 3-dimensional assignment problem~\cite{BDM09}.

Let an $n\times n\times n$ array $C=(c_{ijk})$ be given. For notational
convenience let us introduce {\em planes\/} and {\em lines\/} of 
$C$. We refer to the subarray (matrix) that results from $C$
by fixing the third index $k$ as $k$-planes. 
Lines arise if a subset of two indices is fixed.

The {\em axial 3-dimensional   assignment problem\/} A3AP can be stated 
as follows: 
$$
\min_{\varphi,\psi\in S_n} \sum_{i=1}c_{i\varphi(i)\psi(i)}.
$$

In the {\em planar 3-dimensional assignment problem\/} P3AP the goal is to
find $n$ pairwise disjoint permutations 
$\varphi_1,\varphi_2,\ldots,$ $\varphi_n\in S_n$  
so as to minimize the cost function
\[
\sum_{k=1}^{n}\sum_{i=1}^nc_{i\varphi_k(i)k}
\]
where two permutations $\varphi,\psi\in S_n$ are said to be disjoint if
$\varphi(i)\ne \psi(i)$ for all $i\in \{1,\ldots,n\}$.

Both problems can be formulated as integer linear programs. Below the 
ILP formulation of the P3AP is given. 
\begin{align}\label{p3ap:c}
	\min\ &\sum_{i=1}^n \sum_{j=1}^n \sum_{k=1}^n c_{ijk}x_{ijk}\\ 
	\textrm{s.t.}\ \ &\sum_{k=1}^n x_{ijk}=1 \qquad \qquad i,j=1,2,\ldots,
        n 
\label{p3ap:k}\\
	&\sum_{i=1}^n x_{ijk}=1 \qquad  \qquad j,k=1,2,\ldots, n \label{p3ap:i} \\
	&\sum_{j=1}^n x_{ijk}=1 \qquad  \qquad i,k=1,2,\ldots, n \label{p3ap:j} \\
	&x_{ijk}\in \{0,1\}  \qquad \qquad i,j,k=1,2,\ldots,n.\label{p3ap:0}
\end{align}
Let $c(x)=\sum_{i=1}^n \sum_{j=1}^n \sum_{k=1}^n c_{ijk}x_{ijk}$ denote the
cost of the assignment $x$.
\medskip

Both the A3AP and the P3AP are known to be NP-hard already for 0-1 cost
arrays. For the axial case this has been proved by Karp~\cite{K72} and for 
the planar case by Frieze~\cite{F83}.

The {\em $p$-layer planar 3-dimensional assignment problem\/}, $p$-P3AP for
short, is the following variant/generalization of the P3AP. Given an 
$n\times n\times p$ cost array $C$ where $p\in\{2,\ldots,n\}$, the goal is
to find $p$ pairwise disjoint permutations $\varphi_1,\ldots,\varphi_p$
so as to minimize the cost function
\[
\sum_{k=1}^{p}\sum_{i=1}^nc_{i\varphi_k(i)k}.
\]
Note that for $p=n$ the P3AP results.

The associated integer linear program arises from the one for the 3PAP
by replacing the equality constraints in \eqref{p3ap:k} by
$\sum_{k=1}^p x_{ijk}\leq 1$ and setting the upper limit for the index 
$k$ in the objective function \eqref{p3ap:c} and in the constraints
\eqref{p3ap:i}, \eqref{p3ap:j} and \eqref{p3ap:0} to $p$.

Frieze~\cite{F83} proved that the {\sc DISJOINT MATCHING} 
problem described below 
is NP-complete.
\begin{addmargin}[1em]{1em}
\vspace{5pt}DISJOINT MATCHINGS (DM)\vspace{1pt}\\
\textbf{Input:} Disjoint finite sets $Q,S$ with $|Q|=|S|$ and sets $A_1, A_2 \subseteq P\times Q$.\vspace{1pt} \\
\textbf{Question:} Do there exist matchings $M_i\subseteq A_i$, for $i=1,2$ such that $M_1\cap M_2=\emptyset$?\vspace{4pt}
\end{addmargin}

It immediately follows that the $p$-layer planar 3-dimensional assignment 
problem is NP-hard for every fixed $p\ge 2$.
\smallskip

For the P3AP both exact and heuristic approaches have been devised,
see e.g.\ ~\cite{M96,MM94,V67}.  The polyhedral structure of the P3AP 
has been studied in~\cite{AMM06,EBG86,EV96}. Frieze and Sorkin~\cite{FS13} 
establish high probability lower and upper bounds for the 
optimal objective function value under the assumption that the costs 
$c_{ijk}$ are independent, identically distributed exponential 
random variables with parameter 1.
The P3AP has applications in many areas such as timetabling, rostering and
satellite launching. For references and more details 
see~the survey by Spieksma~\cite{Sp00} and the monograph by Burkard et al.~\cite{BDM09}. 

The $p$-P3AP has received less attention in the literature than 
the P3AP. Gimadi and Glazkov~\cite{GG06} suggested an approximation algorithm
for the $p$-P3AP and analysed its performance for the case where the costs 
are independent, identically distributed uniform random variables.
Yokoya et al.~\cite{YDY11} present an algorithm for the $p$-P3AP 
(called repeated assignment problem in their paper) that outperforms
general ILP-solvers for instances with small $p$.
\smallskip

{\bf Polynomially solvable special cases.} 
For the axial 3-dimensional  several polynomially solvable 
special cases are known, see~\cite{BDM09}. We are not 
aware, however, of any result in the literature on non-trivial polynomially 
solvable cases of the planar 3-dimensional assignment problem.
For the $p$-P3AP the situation is hardly better. The only known special case
concerns the case when all $k$-planes of $C$ are the same, i.e.\,
$c_{ij1}=c_{ij2}=\ldots=c_{ijp}$. In this case the $p$-P3AP reduces to finding
a $p$-factor with minimum cost on a bipartite graph (the $p$-factor can then be
decomposed into $p$ pairwise disjoint perfect matchings/assignments) and is
solvable by network flow methods.
\smallskip

Many hard combinatorial optimization problems become efficiently solvable 
for Monge-like cost structures. Examples include the  
traveling salesman problem~\cite{GLS85}, economic lot-sizing 
problems~\cite{AP93} and the axial d-dimensional assignment and transportation 
problems~\cite{BKR96}. 

The most famous Monge structures are {\em Monge matrices\/} 
which arise in many areas
of mathematics~\cite{BKR96}.
\begin{definition}
An $n\times n$ matrix $M=(m_{ij})$ is 
called Monge matrix if the following property holds
\begin{equation}\label{eq:monge}
m_{ij}+m_{kl}\leq m_{il}+m_{kj} \qquad\text{ for all } 1\le i<k \le n, 1\le j<l
\le n.
\end{equation}
\end{definition}

{\em Monge arrays\/} are a generalization of the concept of Monge matrices to 
$d\ge 3$ dimensions. A $d$-dimensional array is a Monge array if and only if
all its $2\times 2$ submatrices are Monge matrices. More details are provided
in Section~\ref{sec:prel}.

The following folklore result for the linear assignment problem LAP and the 
axial 3-dimensional assignment problem A3AP is well known.

\begin{observation}\label{obs:mongeass} 
Let $\varepsilon_n$ denote the identity permutation 
with $\varepsilon_n(i):=i$.
\begin{itemize}
\item[(i)] $\varepsilon_n$ is an optimal solution for the LAP on $n\times n$
  Monge matrices.
\item[(ii)] $\varphi=\varepsilon_n$ and $\psi=\varepsilon_n$ provides 
an optimal solution for the A3AP on $n\times n\times n$ Monge arrays.
\end{itemize}
\end{observation}

This result motivates to ask whether Monge properties 
are helpful also for the planar 3-dimensional assignment problem 3PAP and its
$p$-layered variant $p$-3PAP.

The main focus in this paper will be on the class of layered Monge arrays
which form a superclass of Monge arrays. The  array $C=(c_{ijk})$
is called {\em layered Monge array} if all its $k$-planes  are Monge matrices.
\medskip

{\bf Organization of the paper.} 
The paper is organized as follows. In Section~\ref{sec:hardness}, we provide
intractability results for the P3AP and the $p$-P3AP. First, it is shown that 
the P3AP remains hard if the cost array $C$ is monotone in all three coordinate
directions (along all lines). 
<As main result of Section~\ref{sec:hardness} it is  shown 
that the P3AP and the $p$-P3AP stay NP-hard on the class of 
layered Monge arrays.

In Section~\ref{sec:struct} the structure of an optimal solution of the 
$p$-P3AP on Monge arrays and on layered Monge arrays is investigated.
For the case of Monge arrays it is shown that there always exists
an optimal  solution of the 2-P3AP which has a  nice block structure. 
The result breaks down already for $p=3$. For the case of layered Monge arrays
it is shown that there always exists an optimal solution of the P3AP 
which can be represented as matrix 
with a bandwidth $\le 4p-3$. This result improves and generalizes
the result on $k$-matchings in a Monge matrix from~\cite{PRW94_2}.
In Section~\ref{sec:algorithm} the bandwidth bound  is 
exploited to provide a dynamic programming algorithm that obtains an optimal
solution of the $p$-P3AP on layered Monge arrays and runs in polynomial time
if $p$ is fixed.


\section{Preliminaries}\label{sec:prel}

\subsection{Basic facts about planar 3-dimensional assignment 
problems}\label{subsec:basics}

A {\em Latin rectangle\/} is an $m\times n$ table $L$ filled with 
integers from the set $\{1,\ldots,n\}$  such that every row and 
every column contains every integer at most once. An $n\times n$ 
Latin rectangle is called {\em Latin square} (of order $n$).
A {\em partial Latin square\/} results if entries of the table can stay 
unfilled.

The following observation is straightforward and well-known.

\begin{observation}\label{obs:latin}
The set of feasible solutions of the $p$-3PAP is in one-to-one 
correspondence with
\begin{itemize}
\item[(i)] the set of $p\times n$ Latin rectangles
\item[(ii)]
the set of partial $n\times n$ Latin squares 
in which each of the integers $1,2,\ldots,p$ appears exactly $n$ times 
and the integers from $\{p+1,\ldots,n\}$ do not appear at all.
\end{itemize}
\end{observation}

To see this, let $x$ be a feasible solution of the $p$-P3AP. 
For (i) place integer $i$ in column $j$ and row $k$  
if and only if
$x_{ijk}=1$. One ends up with a $p\times n$ Latin rectangle.
 
For (ii) place integer $k$ in row $i$ and column $j$ if and only 
if $x_{ijk}=1$. One obtains a partial Latin square of the claimed type. 
For the special case of the P3AP one ends up with a Latin square.
See Figure~\ref{fig:lsq4} for an illustrating example with $n=4$ and $p=3$.
\begin{figure}[h]
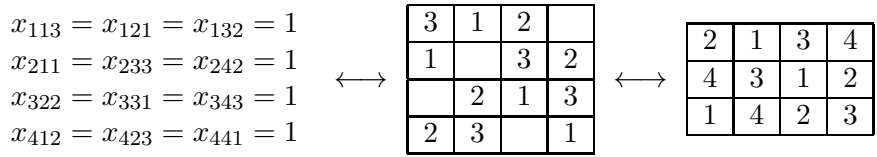

	\centering
	\begin{tabular}{c}
		$x_{113}=x_{121}=x_{132}=1$ \\
		$x_{211}=x_{233}=x_{242}=1$ \\
		$x_{322}=x_{331}=x_{343}=1$ \\
		$x_{412}=x_{423}=x_{441}=1$ \\
	\end{tabular} \ $\longleftrightarrow$ \ 
	\begin{tabular}{|c|c|c|c|}
		\hline 3&1 &2& \\
		\hline 1&&3&2 \\
		\hline &2&1&3 \\
		\hline 2&3&&1 \\ \hline
	\end{tabular}  \ $\longleftrightarrow$ \ 
	\begin{tabular}{|c|c|c|c|}
		\hline 2&1&3&4 \\
		\hline 4&3&1&2 \\
		\hline 1&4&2&3\\\hline
	\end{tabular}
	\caption{Example for partial Latin square/Latin 
rectangle representation}
	\label{fig:lsq4}
\end{figure}

The partial Latin square representation uses an $n\times n$ table regardless
of the value of $p$. For small values of $p$ the Latin rectangle
representation is more compact and will turn out to be useful in 
Section~\ref{sec:struct}.
\medskip

The following observation is straightforward, but turns out to be useful, for
example for showing that assuming monotonicity of $C$ does not make the P3AP
and the $p$-P3AP any easier (see Section~\ref{sec:hardness}).

\begin{observation}\label{obs:shift}
Let $C$ be an $n\times n\times p$ array and let 
$A=(a_{ij})$ be an $n\times n$ matrix and $B=(b_{ij})$ and $D=(d_{ij})$ be 
$n\times p$ matrices.
\begin{itemize}
\item Let $p=n$. The P3AP instances with cost array $C$ and with cost array 
$C+C'$ where 
\begin{equation}\label{eq:sumdecomp1}
c'_{ijk}=a_{ij}+b_{ik}+d_{jk}  \qquad \text{ for all } \quad
i,j,k\in\{1,\ldots,n\}
\end{equation}
are equivalent.   The objective function value is shifted by
the constant $\alpha=\sum_{i=1}^n\sum_{j=1}^n(a_{ij}+b_{ij}+d_{ij})$.
\item The $p$-P3AP instances with cost array $C$ and with cost array $C+C''$
  where 
\begin{equation}\label{eq:sumdecomp2}
c''_{ijk}=b_{ik}+d_{jk}  \qquad \text{ for all } \quad
i,j\in\{1,\ldots,n\}, k\in\{1,\ldots,p\}
\end{equation}
are equivalent. The objective function value is shifted by the constant
$\beta=\sum_{i=1}^n \sum_{k=1}^p(b_{ik}+d_{ik})$.
\end{itemize}
\end{observation}

We mention that arrays that fulfill the property \eqref{eq:sumdecomp1} 
have been named sum-decomposable with parameters $d=3$ and $s=2$ in~\cite{CK14}.
There it was proved that all feasible solutions of the P3AP
have the same cost if and only if the cost array $C$ is sum-decomposable
with parameters $d=3$ and $s=2$. The cost arrays that fulfill
\eqref{eq:sumdecomp2} play an analogous role for the $p$-P3AP for the 
$p<n$ case.

\subsection{Cost arrays with Monge properties}\label{subsec:monge}

The notion of Monge matrices can be generalized to $d>2$ dimensions. 
\begin{definition}\label{MonDef}
An $n_1\times n_2 \times \cdots \times n_d$ $d$-dimensional array 
$C=(c_{i_1i_2\cdots i_d})$ is called \textbf{Monge array} if for 
all $i_k=1,\ldots,n$ and $j_k=1,\ldots,n$, $k=1\ldots,d,$ we have
\begin{equation}c_{s_1s_2\cdots s_d}+c_{t_1t_2\cdots t_d}
\le c_{i_1i_2\cdots i_d}+c_{j_1j_2\cdots j_d},
\end{equation}
where $s_k=\min\{i_k,j_k\}$ and $t_k=\max\{i_k,j_k\}$, $k=1\ldots,n$. 
\end{definition}

Note that a two dimensional Monge array is a Monge matrix. It is well known
that a $d$-dimensional array $C$ is a Monge array if and only if every
two-dimensional subarray (matrix) of $C$ corresponding to fixed values for
$d-2$ of the $d$ indices in $C$ is a Monge matrix, see~\cite{AP89}.

Next we define a special class of Monge arrays.
\begin{definition}\label{defDis}
An $n\times n \times \cdots \times n$ $d$-dimensional array 
$C=(c_{i_1i_2\cdots i_d})$ is called {\em distribution array\/} if
\begin{equation}
c_{i_1i_2\cdots i_d}=-\sum_{j_1=1}^{i_1}\sum_{j_2=1}^{i_2}\ldots\sum_{j_d=1}^{i_d}
p_{j_1j_2\cdots j_d}
\end{equation}
where $p_{j_1j_2\cdots j_d}\geq 0$ for all 
$(j_1,j_2,\ldots, j_d)\in\{1,\ldots,n\}^d$.
\end{definition}

It is easy to show that every distribution array is a Monge array.
Next we consider a superclass of the class of Monge arrays.
\begin{definition}
Let $C=(c_{ijk})$ be a three-dimensional $n\times n\times p$ array such that 
for every $k$, $1\leq k\leq p$, the matrix $M^k=(m^k_{ij})$ where 
$m^k_{ij}=c_{ijk}$ is a Monge matrix. Then $C$ is called an {\em array with
  Monge layers\/}, or {\em layered Monge array\/} for short.
\end{definition}

\section{Intractability results}\label{sec:hardness}

Some hard combinatorial optimization problems become easy if the cost structure
fulfills a monotonicity property. Unfortunately, monotonicity does not help for
the P3AP.

\begin{definition}
An $n_1\times n_2\times n_3$ cost array $C$ is called triply graded if
$C$ is monotone increasing along all lines of $C$.
\end{definition}

\begin{theorem}\label{thm:mono}
The P3AP and the $p$-P3AP stay NP-hard on the class of triply graded cost 
arrays.
\end{theorem}

\begin{proof}
Let a $n\times n\times p$ cost array $C$ with $p\in\{2,\ldots,n\}$ be given.
We will show that $C=(c_{ijk})$ can be turned into a
triply graded cost array $\widetilde{C}$ such 
that the order of the feasible solutions of the $p$-3PAP 
with respect to the objective function  value does not change.

Let 
$$m\ge |\max_{i,j,k}{c_{ijk}}-\min_{i,j,k}{c_{ijk}}|.$$
Define 
$\widetilde{c}_{ijk}=c_{ijk}+(i+j+k)m$.
It follows from Observation~\ref{obs:shift} that 
the $p$-P3AP instances with cost arrays $C$ and $\widetilde{C}$ are 
equivalent. Moreover, it can be seen easily that $\widetilde{C}$
is monotone increasing along all lines and hence triply graded.
\end{proof}

Note that the same approach also works if the direction of monoticity is not
the same in all three coordinate directions. Next we turn to the case of 
layered Monge arrays.

\begin{theorem}\label{thm:monge1}
The P3AP stays NP-hard on the class of layered Monge cost arrays.
\end{theorem}
\begin{proof}
Consider the $n\times n$ matrix $M=(m_{ij})$ with $m_{ij}=-(i+j)^2$ (or
alternatively $m_{ij}=4n^2-(i+j)^2$ if one prefers to deal with nonnegative cost
arrays). It is easy to check that $M$ is a Monge matrix. Let $F=(f_{ijk})$ be 
the $n\times n\times n$ cost array obtained from $M$ by setting 
$f_{ijk}=m_{ij}$ for all $i,j,k\in\{1,\ldots,n\}$. Since $M$ is a Monge
matrix, $F$ is a layered Monge array.

Let $C$ be an arbitrary $n\times n\times n$ 0-1 array. Consider the $n\times
n\times n$ cost array $C'=(c'_{ijk})$ with $c'_{ijk}=f_{ijk}+c_{ijk}$. It is easy
to check that the layered Monge array property is inherited to $C'$ from $F$.

It follows from Observation~\ref{obs:shift} that the P3AP on cost matrix 
$C$ is equivalent to the P3AP on cost array $C'$. Since the P3AP is NP-hard for
general 0-1 cost arrays it follows that the P3AP stays hard when restricted to
layered Monge arrays.
\end{proof}

\begin{theorem}\label{thm:monge2}
The $p$-P3AP stays NP-hard on the class of layered Monge cost arrays.
\end{theorem}

\begin{proof}
The proof builds on the idea used in the proof of Theorem~\ref{thm:monge1}.
We make again use of the $n\times n$ Monge matrix $M=(m_{ij})$
with $m_{ij}=-(i+j)^2$. Let $n'=2n$ and expand
$M$ into an $n'\times n'$ matrix $M'$ as follows
$$
M'=\begin{pmatrix}
	M & Y \\
	Y^t & Z\\
\end{pmatrix}
$$
where $Z$ is the $n\times n$ zero matrix and 
$Y=(y_{ij})$ is an $n\times n$ matrix 
with $y_{ij}=i\cdot n$ for all $i,j\in \{1,\ldots,n\}$.
It is easy to check that $M'$ is again a Monge matrix.

Let an instance of the P3AP with an $n\times n\times n$ 0-1 cost array $C$ 
be given.  Expand $C$ into an $n'\times n'\times n$ array $\widehat{C}=
(\widehat{c}_{ijk})$ by
defining 
$$\widehat{c}_{ijk}=
\begin{cases} 
c_{ijk} & \text{for }  i,j,k\in\{1,\ldots,n\}\\
0 & \text{for } i,j\in\{n+1,\ldots,n'\}, k\in\{1,\ldots,n\}
\end{cases}
$$
Now create the $n'\times n'\times n$ cost arrays $F'=(f'_{ijk})$ and 
$C'=(c'_{ijk})$ with 
$$f'_{ijk}=m'_{ij} \quad \text{ for all } i,j\in\{1,\ldots,n'\}, k\in  \{1,\ldots,n\}
$$
and
$$c'_{ijk}=f_{ijk}+\widehat{c}_{ijk}  \quad \text{ for all } i,j\in\{1,\ldots,n'\}, 
k\in  \{1,\ldots,n\}.
$$

It is easy to check that  $C'$ is a layered Monge array.
Now consider the instance $I$ of the  $p$-P3AP with cost array $C'$ array and 
$p=n$.

Note that the cost entries in the $Y$ and the $Y^t$ block of $M'$ (which carry
over to the $k$-planes of $F'$ and $C'$) are positive numbers $\ge
n$ while the entries in $Z$ are zero and the entries in $M$ are negative.
Thus an optimal solution of the $p$-P3AP will contain only
elements $(i,j,k)$ for which either $i,j\in\{1,\ldots,n\}$ or
$i,j\in\{n+1,\ldots,n'\}$ holds when such solutions exist (which is easily seen
to be the case).

Hence an optimal solution of the P3AP with cost array $C$ can be obtained 
from an optimal solution of the $p$-P3AP instance $I$  
by dropping all 3-tuples in the solution that involve indices $>n$ .
The result now follows from the NP-hardness of the P3AP for 0-1 cost arrays.
\end{proof}

Note that the cost arrays used in the hardness reductions in the proofs for
Theorems~\ref{thm:monge1} and~\ref{thm:monge2} do not belong to the class of
Monge arrays. The complexity status of the P3AP and of the $p$-P3AP on the
class of Monge arrays and its subclass distribution arrays remains open.


\section{The structure of the optimal solution of the 
$\mathbf p$-P3AP on layered Monge arrays}\label{sec:struct}

Note that in the hard instance constructed in the proof of 
Theorem~\ref{thm:monge2} 
the number of layers $p$ is of the order of $n$. This leaves the complexity
status of the $p$-P3AP for layered Monge arrays unsettled when $p$ is a 
constant. In this section we present results on the structure of the optimal
solution of the $p$-P3AP on layered Monge arrays. The bandwidth result
shown in Theorem~\ref{thm:band}  will be used
in Section~\ref{sec:algorithm} to provide an algorithm that solves the
$p$-P3AP on layered Monge arrays in polynomial time if $p$ is fixed.

\subsection{Block structure result for the 2-P3AP}

In this subsection we will prove that there always
exists an optimal solution of the 2-P3AP on layered Monge arrays which fulfills
a nice block structure. To formulate this result and to prove it
we will need the following  definitions. It will be convenient to 
represent the feasible solutions of the 2-P3AP by Latin rectangles, 
cf.\ Section~\ref{subsec:basics}.

\begin{definition}
A partition of a Latin rectangle into {\em blocks\/} is a partition 
of the columns into sets of adjacent 
columns such that such each set of $m$ columns contains only $m$ 
different integers, and is minimal with respect to that property. 
A block which spreads from $k$-th to the $l$-th column, $k<l$ and 
contains only integers from the set $\{k,\ldots,l\}$ 
is called {\em normalized block\/}.
\end{definition}

As an illustrative example consider the $2\times 12$ Latin rectangle 
below. It is partitioned into 4 blocks. The second block which contains the
numbers 3, 4 and 5 is the only normalized block.

$$\begin{tabular}{|cc|ccc|cccc|ccc|}\hline 6 & 2 & 4& 
3 & 5& 8& 11& 12& 1& 7& 10& 9\\ \hline 2& 6& 
 3& 5& 4&12& 1& 11&8&10& 9& 7\\ 
\hline 
\end{tabular}$$

Note that if the block $B$ is a normalized block which consists of 
 the $m$ columns $j,\ldots, j+m-1$ the following property is 
fulfilled for all $i=1,\ldots,m-1$
\begin{itemize}
\item [(*)]
The first $i$ columns of $B$ contain an integer $x>i+j$ and the
last $i$ columns of $B$ contain an integer $x<j+m-i$.
\end{itemize}

The local operation of a swap that exchanges two integers in a row of an
$p\times n$ Latin rectangle will play a fundamental role in what follows.

\begin{definition}
Let a feasible solution $x$ of the $p$-P3AP be given and let 
$r,q\in \{1,\ldots,n\}$, $r<q$. The operation that exchanges the 
positions of $r$ and $q$ in a row $k$ of the Latin rectangle 
representation of $x$ is referred to as
a {\em swap} and denoted as {\sc SWAP}$(r,q,k)$. 
The swap is called {\em feasible\/} if the newly obtained
solution  is feasible. The swap is called  {\em non-increasing\/} 
if the newly obtained obtained solution $x'$ has cost $c(x')\le c(x)$.
\end{definition}

\begin{observation}
Let $C$ be an $n\times n\times p$ layered Monge array. 
Then the swap {\sc SWAP}$(r,q,k)$ is non-increasing if $r$ is placed to the
right of $q$.
\end{observation}

\begin{proof}
Let $q$ be placed in column $j$ of row $k$ and $r<q$ be placed in column $s>j$,
i.e.\ $x_{qjk}=1$ and $x_{rsk}=1$.
Since $C$ is a layered Monge array, the following property is fulfilled
$c_{rjk}+c_{qsk}\leq c_{rsk}+c_{qjk}$. It follows that exchanging the position
of $r$ and $q$ in row $k$ will lead to a new (not necessarily feasible)
solution $x'$ with $x'_{rjk}=1$ and $x'_{qsk}=1$ such that $c(x')\le c(x)$.
\end{proof}

$$ r<q\quad \Longrightarrow\quad \begin{tabular}{|c|c|c|c|c|c|}\hline  \ \ 
&$r$&\ \ &$q$&\ \ &\ \ \\ \hline \ &\ &\ &\ &\ &\ \\ \hline \end{tabular} \ 
\leq \ \begin{tabular}{|c|c|c|c|c|c|}\hline  \ \ &$q$&\ \ &$r$&\ \ &\ \ \\ 
\hline \ &\ &\ &\ &\ &\ \\ \hline \end{tabular}$$

\begin{proposition}\label{prop:2layers}
For the 2-P3AP with an $n\times n\times 2$ cost array $C$ which is a layered
Monge array, there always exists an optimal solution such that its
corresponding Latin rectangle decomposes into normalized blocks 
with 2 or 3 columns.
\end{proposition}

Our proof  is based on an extensive case distinction and is 
unfortunately not elegant. To improve the readability of the paper, we have
decided to move the proof to an appendix (which will be included in the second
version of the paper).

For the special case of distribution arrays 
we have a more elegant proof which is based on the fact that for distribution 
arrays there exists an explicit formula for the cost of a feasible 
solution which in terms of the entries of the density matrix. 
As the result above is more general, we decided
to omit the proof for the special case.

\subsection{Example with a single large block for $\mathbf{p=3}$}

It would be nice if the block structure result for the 2-P3AP from
Proposition~\ref{prop:2layers} carried over to the 3-P3AP. Unfortunately, this
is not the case, not even for the subclass of distribution arrays 
as is demonstrated by the following counterexample in which the unique optimal
solution consists of a single large block.

Let $n=10$ and $p=3$. We consider the distribution matrix $C$ which is
generated by the matrix $P=(p_{ijk})$ with 
$$	 p_{ij1}=\left\{\begin{array}{r@{\,\quad }l}
		100 & \textrm{if $i=j=7$}\\
		1 & \textrm{otherwise}
	\end{array}\right.\quad
         p_{ij2}=\left\{\begin{array}{r@{\,\quad }l}
		10a & \textrm{if $i=4$ and $j=5$}\\
		10a^3 & \textrm{if $i=9$ and $j=10$}\\
		a_j & \textrm{otherwise}
	\end{array}\right.
$$
and $p_{ij3}=a^a$ for all $i,j\in\{1,\ldots,10\}$
where $a_n=(1,1,a,a,a,a^2,a^2,a^3,a^3,a^3)$ and $a$ is some large number.
For notational convenience let us refer to the $10\times 10$ matrix that 
corresponds to the $k$-plane of $P$ as $Q_k$.

We claim that the following $3\times 10$ Latin rectangle $L$ 
\begin{equation}\label{eq:opt}
	\begin{tabular}{|cccccccccc|}
		\hline 3&4&1&2&6&5&8&10&7&9 \\ \hline
		2&1&4&5&3&7&6&9&10&8 \\\hline
		1&2&3&4&5&6&7&8&9&10 \\ \hline
	\end{tabular}
\end{equation}
is the unique optimal solution for the 3-P3AP instance with cost array $C$.
This can be checked by computational means. Below we provide the rationale
behind why $L$ shows up as unique optimal solution.

Consider the $10\times 10\times 2$ subarray $C_2$ of $C$ 
where the the $k$-plane of $C$ with $k=1$ is dropped.
We claim that the following Latin rectangle provides an optimal solution of the
2-P3AP on $C_2$ 
\begin{equation}\label{eq:3sol}
	\begin{tabular}{|cc|ccc|cc|ccc|}
		\hline 2&1&4&5&3&7&6&9&10&8 \\
		\hline 1&2&3&4&5&6&7&8&9&10 \\ \hline
	\end{tabular}
\end{equation}
This is easy to check. Note that the entries of $Q_3$ are much larger than the
entries in $Q_2$ and $Q_1$. It can easily be checked that it thus pays off
to choose the overall best permutation for any Monge matrix, namely the 
identity permutation, as solution for the $k$-plane with $k=3$. Choosing any
other permutation more will be lost than can be won for the $k$-planes with
$k=1$ and $k=2$. Hence the second row of \eqref{eq:3sol} is the
identity permutation. 

As a consequence of the structure of the vector $a_n$ and its role in $Q_3$
\eqref{eq:3sol} splits into blocks of sizes 2,3,2 and 3 respectively.
The two entries of $Q_2$ which involve the multiplicative factor 10  
determine the type of the two blocks of size 3 in \eqref{eq:3sol}.

 Since $Q_1$ has much smaller entries than $Q_2$ and $Q_3$, the optimal solution
 of the $10\times 10\times 3$ P3AP with cost array $C$ will have
 \eqref{eq:3sol} as last two rows. This is  true since every set of 
Latin rows can be completed to Latin square~\cite{H45}. 

Because $p_{1,7,7}=100$, the rightmost four columns of the first row in
the optimal Latin rectangle $L$ will be filled with integers 
greater or equal than 7, if possible. There indeed exist three ways to achieve
this goal and we take the one with the lowest cost, i.e., the block
(8,10,7,9), and fill the rest of the first row of the Latin rectangle in 
an obviously unique optimal manner. The Latin  rectangle $L$ results.
\medskip

As the Latin rectangle $L$ consists of a single block of size $n=10$, this
destroys any hope for a result along the lines of
Proposition~\ref{prop:2layers} for $p\ge 3$.

The same approach can be used to construct examples where an 
arbitrarily large block arises in the optimal solution. Just 
add an arbitrary number of $2\times2$ blocks in the middle of the 
candidate solution \eqref{eq:3sol}.

The construction above can be carried over $p>3$. Add
more $k$-planes on top of $P$ with much smaller values than in the first
three $k$-planes. Then the optimal solution will have a 
Latin rectangle representation such that the first three rows are as in the
example provided for $p=3$. (Use Hall's theorem~\cite{H45}.)

\subsection{Bandwidth result for the  $\mathbf p$-P3AP}

Recall that for the linear assignment problem on a Monge matrix the identity
permutation provides an optimal solution, cf.\ Observation~\ref{obs:mongeass}.

A feasible solution of the $p$-P3AP is a set of $p$ pairwise 
disjoint permutations (assignments). Hence, one might expect that 
for the $p$-P3AP on a layered Monge array the filled cells of the 
partial Latin square representing the optimal 
solution tend to group around the main diagonal. It is shown below 
that this is indeed the case.


\begin{theorem}\label{thm:band}
Let an  instance $I$ of the $p$-P3AP with the $n\times n\times p$ layered Monge 
cost array $C$ be given. There exists an optimal solution $x$ for $I$
such that the partial Latin sq<uare $L$ that represents $x$ has 
the following property:
\begin{itemize}
\item[(P)] Whenever the cell $L(i,j)$ is filled, we have 
$|i-j|\leq 2p-2$.
\end{itemize}
\end{theorem}

\begin{proof} 
Consider an optimal solution $x$ and its representation as partial 
Latin square $L$. We start at the upper right corner and shift a line parallel 
to the main diagonal towards the center of the partial Latin square 
until one hits for the first time a filled cell. Choose such a cell and
call it \emph{pivotal cell}. Suppose that the pivotal cell is $L(i,j)$ and
contains the integer $k$, see Figure~\ref{fig:sw} for an illustration.  
\begin{figure}[h]
	\centering
	\includegraphics{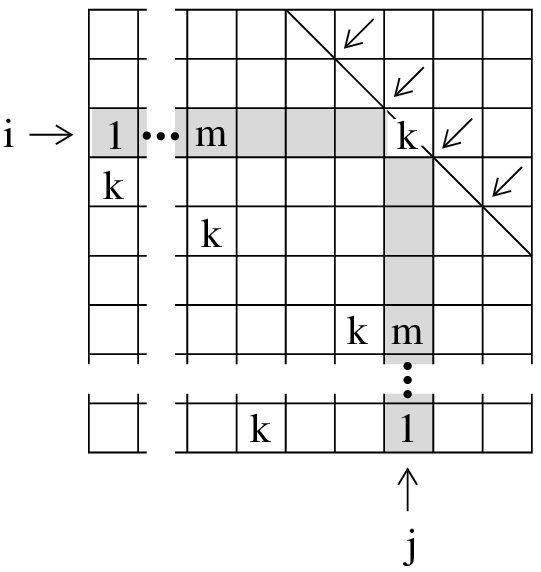}
	\caption{Pushing elements closer to the diagonal}
	\label{fig:sw}
\end{figure}
We define four sub-rectangles of the partial Latin square $L$ as follows
depending in which of the four quadrants around $L(i,j)$ the cells are located.
\begin{itemize}
\item $Q_{\text{la}}$: Contains all cells $(r,q)$ with $r<i$ and $q<j$.
\item $Q_{\text{lb}}$: Contains all cells $(r,q)$ with $r>i$ and $q<j$.
\item $Q_{\text{ra}}$: Contains all cells $(r,q)$ with $r<i$ and $q>j$.
\item $Q_{\text{rb}}$: Contains all cells $(r,q)$ with $r>i$ and $q>j$.
\end{itemize}
Let us refer to $Q_{\text{lb}}$ as {\em candidate area}.
If there exists a cell $L(q,r)$ in the candidate area that is filled 
with the integer $k$ while the cells $L(i,r)$ and $L(q,j)$ are empty, we can 
fill the cells $L(i,r)$ and $L(q,j)$ with the integer $k$ and 
restore the cells $L(i,j)$ and $L(q,r)$ to the empty state by deleting $k$ from
them. The new solution $x'$ is feasible again and since
each 3-plane of the cost array $C$ is a Monge matrix, it follows
that 
$$
c_{irk}+c_{qjk}\le c_{ijk}+c_{qrk},
$$
and hence the cost does not increase if we move to the new solution.

Next we examine for which $i$ and $j$ it is always possible to perform a move
of the type described above. First we count how many cells in the candidate
area are filled with the integer $k$. Since there are no filled cells 
in $Q_{\text {ra}}$ by assumption,
it follows that $k$ occurs in $Q_{\text{la}}$ exactly $i-1$ times.
Analogously, it follows that $k$ occurs in $Q_{\text{rb}}$ exactly $n-j$ times.
Consequently, $k$ occurs in the the candidate area exactly 
$n-(i-1)-(n-j)-1=j-i$ times.

Note that filled cells in row $i$ left of the pivotal cell
and filled cells in column $j$ below of the pivotal 
element (see the gray area in Figure~\ref{fig:sw})
can lead to the in the situation
that a candidate integer $k$ cannot be used to perform a move of the
type described above. Note that there are exactly $p-1$ filled cells in row $i$
left of $L(i,j)$ and exactly $p-1$ filled cells in column $j$ below $L(i,j)$.
Therefore, if \[ j-i>2p-2, \] holds, we always can choose an integer such the
performed move is feasible.

Similarly, if we consider the elements below the diagonal,
we get that a feasible move can be performed  if 
$i-j> 2p-2$. Therefore there always exists an optimal solution 
such that for all filled cells $L(i,j)$ we have 
\[ |i-j|\leq 2p-2. \]
\end{proof}

Note that property (P) implies that the partial Latin square $L$ is a matrix
with bandwidth $\le 4p-3$. It is natural to ask whether this bound is tight.
For $p=2$ and odd $n$ it is easy to see that this is the case. It suffices to 
choose the two 3-planes of the cost array $C$ such that in the optimal partial
Latin square all diagonal entries will be filled by 1-entries which is easy to
achieve.

We conjecture that the bound $4p-3$ is tight for infinitely many $p$. 
More specifically, we have constructed a class of instances  for which we
conjecture that the bound is achieved when $p$ is prime. We have checked the
correctness for values of $p$ up to 31 by computational means. Work on the
proof is in progress.

\section{A dynamic programming algorithm for the 
$\mathbf p$-PLAP on layered Monge arrays}\label{sec:algorithm}

Based on Theorem~\ref{thm:band} we can devise a dynamic programming algorithm 
that finds an optimal solution of the $p$-P3AP on layered Monge arrays. The
algorithm's running time will be linear in $n$ and exponential in $p$. 

Our dynamic programming algorithm will start with the empty partial Latin
square and will then fill the rows of the partial Latin square with 
integers iteratively, going from the top to the bottom. 
Integers will be placed in the band around the diagonal according to 
Theorem~\ref{thm:band} while at the same time taking the feasibility aspect
into account.

The state space of our dynamic program is defined as follows.
Each state corresponds to a partial Latin square. 
Let $\mathcal{S}_i$ denote the set of states built in the $i$-th step, that is, 
the set of partial Latin squares with the first $i$ rows filled with integers 
1 to $p$ while the remaining rows are empty. 
Furthermore, $\mathcal{S}_i$ will fulfill the following additional properties:
\begin{enumerate}
	\item[i)] The partial Latin squares that correspond to a state in 
$\mathcal{S}_i$ contain exactly $p$ integers in the columns 1 to $i+2p-2$ 
and no integers in the columns $i+2p-1$ to $n$. 
	\item[ii)] To every state $S$ from $\mathcal{S}_i$ we associate a $(4p-4)$-tuple of subsets of $\{1,\ldots,p\}$ that correspond to integers that occur in columns  $i-2p+3$ to $i+2p-2$ of $S$, respectively. No two states in $\mathcal{S}_i$ have the same corresponding $(4p-4)$-tuples.
\end{enumerate}
We start the algorithm with $\mathcal{S}_0=\{S_0\}$, where $S_0$ is an empty
$n\times n$ table. In the $i$-th step, for each state $S_{i-1}$ from
$\mathcal{S}_{i-1}$ do the following. Place integers 1 to $p$ into the columns
$i-2p+2$ to $i+2p-2$ of the row $i$, in every possible way. Let $S_{i}$ be one
such table. If $S_{i}$ is not a partial Latin square or does not satisfy the
condition i),  discard it. Otherwise, check the 
$(4p-4)$-tuple defined in ii). If there is no state in $\mathcal{S}_i$ with the
same corresponding $(4p-4)$-tuple, put $S_i$ into $\mathcal{S}_i$. If there 
already exists such state, then compare the corresponding cost values 
of the two states, and put in $\mathcal{S}_i$ only the state  with 
the smaller cost. In doing so, condition ii) will be ensured. At the 
end of the step $n$, $\mathcal{S}_n$ will consist of only one state, which will
correspond to an optimal solution. 

The size of $\mathcal{S}_i$ does not depend on $n$. Also, the number of ways
one can fill a new row, and the time needed to check whether there already 
exists a state in $\mathcal{S}_i$ with the same $(4p-4)$-tuple than 
RyRythe $(4p-4)$-tuple
of the currently considered partial solution.
 does not depend on $n$. Hence the time complexity of the algorithm is 
$O(f_pn)$, where $f_p$ is an exponential function in $p$. 
Note that to build $\mathcal{S}_i$ we only need
$\mathcal{S}_{i-1}$, hence space complexity does does not depend on $n$ either.

\begin{corollary}\label{cor:fixedp}
The $p$-P3AP for layered Monge cost arrays is fixed parameter tractable with
respect to the parameter $p$.
\end{corollary}


\section*{Acknowledgement}
This research was supported by the Austrian Science Fund (FWF): 
W1230, Doctoral Program ``Discrete Mathematics''.

\bibliographystyle{plain}

\begin{appendix}
\section{Appendix: Proof for Proposition~\ref{prop:2layers}}

The proof will be part of version 2.

\end{appendix}

\end{document}